\documentclass[11pt,a4paper]{article}

\usepackage[centering, footskip=.65in, margin=1.3in]{geometry}

\usepackage{datetime}
\usepackage{float}

\usepackage{eucal}

\usepackage[all]{xy}

\usepackage{amssymb, amsmath, amsthm}

\usepackage{tikz}

\numberwithin{equation}{section}

\newlength{\bibitemsep}
\newlength{\bibparskip}
\let\oldthebibliography\thebibliography
\renewcommand\thebibliography[1]{%
  \oldthebibliography{#1}%
  \setlength{\parskip}{\bibitemsep}%
  \setlength{\itemsep}{\bibparskip}%
}

\usepackage[pdftex]{hyperref}

\hypersetup{
  colorlinks = true,
  linkcolor = black,
  urlcolor = gray,
  citecolor = black,
}

\renewcommand{\emptyset}{\ensuremath{\text{\upshape\O}}}

\renewcommand{\phi}{\varphi}

\newcommand{\bbC}{\mathbb{C}}
\newcommand{\bbF}{\mathbb{F}}

\newcommand{\bbQ}{\mathbb{Q}}

\newcommand{\bbR}{\mathbb{R}}
\newcommand{\bbZ}{\mathbb{Z}}

\newcommand{\Qtr}{\bbQ^\mathrm{tr}}

\newcommand{\rmA}{\mathrm{A}}

\newcommand{\rmD}{\mathrm{D}}
\newcommand{\rmF}{\mathrm{F}}
\newcommand{\hatF}{\mathrm{F}^{\wedge}}
\newcommand{\hatFQ}{\mathrm{FQ}^{\wedge}}
\newcommand{\hatGr}{\mathrm{Gr}^{\wedge}}
\newcommand{\hatZ}{\widehat{\mathbb{Z}}}

\newcommand{\rmK}{\mathrm{K}}

\newcommand{\rmS}{\mathrm{S}}
\newcommand{\rmW}{\mathrm{W}}

\newcommand{\Gr}{\mathrm{Gr}}
\newcommand{\AS}{\mathrm{AS}}

\newcommand{\GW}{\mathrm{GW}}

\DeclareMathOperator{\FQ}{FQ}
\DeclareMathOperator{\id}{id}
\DeclareMathOperator{\gr}{gr}
\DeclareMathOperator{\Tr}{Tr}

\DeclareMathOperator{\Mor}{Mor}

\DeclareMathOperator{\Aut}{Aut}
\DeclareMathOperator{\Gal}{Gal}
\DeclareMathOperator{\Orb}{Orb}
\DeclareMathOperator{\Ord}{Ord}

\DeclareMathOperator{\Inv}{Inv}

\DeclareMathOperator{\Fix}{Fix}
\DeclareMathOperator{\Homeo}{Top}

\DeclareMathOperator{\ch}{char}
\DeclareMathOperator{\tr}{tr}

\newtheorem{theorem}{Theorem}[section]
\newtheorem{proposition}[theorem]{Proposition}

\newtheorem{lemma}[theorem]{Lemma}

\newtheorem*{theorem*}{Theorem}

\theoremstyle{definition}

\newtheorem{definition}[theorem]{Definition}
\newtheorem{example}[theorem]{Example}
\newtheorem{examples}[theorem]{Examples}
\newtheorem{remark}[theorem]{Remark}

\title{\bf Artin--Schreier quandles of involutions\\ in absolute Galois groups}

\author{Markus Szymik}

\newdateformat{mydate}{\monthname~\twodigit{\THEYEAR}}
\date{\mydate\today}

\begin{document}

\maketitle


\renewcommand{\abstractname}{\vspace{-\baselineskip}}

\noindent
We introduce a new invariant of fields that refines their real spectrum and is related to their absolute Galois group: the Artin--Schreier quandle. For totally real number fields, it is freely generated---in the variety of profinite involutory quandles---by a Cantor space of indeterminates. For Laurent series fields, we compute it in terms of the Artin--Schreier quandle of the coefficient field. This result, along with other examples, shows that, in general, there are relations.


\section{Introduction}

The real spectrum of a commutative ring is a topological space that enriches its Zariski spectrum by incorporating ideas from the theory of formally real fields. For a field, which has only one prime ideal, the Zariski spectrum is merely a singleton, whereas the real spectrum comprises the space of orderings of the field. From a group-theoretic perspective, we know that we can identify the real spectrum of a field with the space of conjugacy classes of involutions in the absolute Galois group~(see Section~\ref{sec:AS(F)} for a review). This observation allows us, in this paper, to further refine the real spectrum of a field~$F$. Instead of working with equivalence classes in the form of conjugacy classes of involutions, we define the Artin--Schreier quandle~$\AS(F)$ directly, based on the set of all involutions within the absolute Galois group~$\Gal(F)=\Gal(\overline{F}|F)$. A novel observation is that this set not only possesses a topological structure as a subspace of a profinite group but also an algebraic one: Although it is not a subgroup, it is still closed under conjugation. The algebraic theory of quandles is a well-established axiomatic framework for modelling conjugation, and~$\AS(F)$, therefore, inherits the structure of an involutory quandle~(see Section~\ref{sec:involutions} for a review). One of our main results characterises the profinite involutory quandle~$\AS(F)$ in the most fundamental cases:

\begin{theorem} 
The Artin--Schreier quandle~$\AS(\bbQ)$ of the rational number field~$\bbQ$ is a free profinite involutory quandle. A Cantor space of involutions inside the absolute Galois group~$\Gal(\bbQ)$ gives a basis. A similar statement holds for all totally real number fields.
\end{theorem}

We prove this result as Theorems~\ref{thm:AS(QQ)} and~\ref{thm:AS(F)} in the text. Totally imaginary number fields have no involutions in their absolute Galois groups, and their Artin--Schreier quandles are empty.

For more context, recall that Hilbert's 17th problem asked if we could write a real multivariate polynomial that only assumes non-negative values as a sum of squares of rational functions. This problem was solved affirmatively in 1927 by Artin~\cite{Artin:Hilbert17} who, in collaboration with Schreier~\cite{Artin:reell_algebraisch,AS:1,AS:2}, founded the subject of real algebra. In 1979, Coste and Roy introduced the real spectrum of a commutative ring~\cite{CR}. It plays the same role in real algebraic geometry as the ordinary prime spectrum in the usual algebraic geometry. For a field~$F$, the real spectrum has various, seemingly different descriptions: If~$F$ is a number field, we can identify it with the~(finite) set of embeddings of~$F$ into the field~$\bbR$ of real numbers. In general, referring again to the review in Section~\ref{sec:AS(F)} for more detail, it is better to think of it as the set of orderings of~$F$, as the set of ring morphisms~$\rmW(F)\to\bbZ$, where~$\rmW(F)$ is the Witt ring of quadratic forms over~$F$, or---for us---as the set of conjugacy classes of involutions in the absolute Galois group~$\Gal(F)$. The Artin--Schreier quandle replaces the set of conjugacy classes with the collection of all their representatives and non-judgementally embraces the algebraic structure that comes with it. 

The thrust of this idea extends beyond real algebra to the many other contexts of conjugation in algebraic number theory, and it opens new avenues in the flourishing~`arithmetic topology' program initiated by Manin, Mazur, Mumford, and others. The connection to topology comes via an algebraic structure that codifies conjugation: quandles. We refer to Section~\ref{sec:involutions} again for a review of the basic definition, but we point out here that we can think of quandles as generalisations of knots: first considered by Burde~\cite{Burde} in that context, Joyce~\cite{Joyce} and Matveev~\cite{Matveev}, building on earlier work of Waldhausen, showed that knot quandles, a canonical refinement of knot groups, completely classify knots. The knot quandle of a prime knot is the conjugacy class of the meridian with the quandle structure given by conjugation.~(For composite knots, the description is more subtle and not relevant here.) Meridians are generators of the fundamental groups of `small' disks that are transversal to the knot and have been punctured by removing the intersection with the knot. Given the importance of meridians for understanding knot groups and quandles, we are urged to understand the function fields~$F(\!(t)\!)$ of punctured formal disks. The following result is Theorem~\ref{thm:Laurent} in the main text.

\begin{theorem}
For any field~$F$, there is an isomorphism
\[
\AS(F(\!(t)\!))\cong\AS(F)\times\hatFQ_2\{\sigma,\tau\}\cong\AS(F)\times\hatZ
\]
of profinite involutory quandles, where~$\hatFQ_2\{\sigma,\tau\}$ is the free profinite involutory quandle on a set~$\{\sigma,\tau\}$ with two elements, and which, as a profinite space, is homeomorphic to~$\hatZ$. The quandle structure is given by~$a\rhd b=2a-b$ for~$a,b\in\hatZ$.
\end{theorem}

We note that Takahashi~\cite{Takahashi}~(see also Davis--Schlank~\cite{Davis--Schlank}) has associated quandles with Galois covers of arithmetic schemes in a different way than here. We refer to Shusterman~\cite{Shusterman} for another framework that uses quandles and the more general racks in number theory.

This paper is written to be broadly accessible to number theorists and topologists. We discuss involutions and involutory quandles in Section~\ref{sec:involutions}, with the second half of that section describing the free models and covering the topological details needed for the profinite context. We specialise this general theory to Artin--Schreier quandles in Section~\ref{sec:AS(F)}. We describe the Artin--Schreier quandles of number fields and other global fields in Section~\ref{sec:QQ}. The final Section~\ref{sec:Laurent} contains Theorem~\ref{thm:Laurent} and discusses examples of Artin--Schreier quandles that have relations.


\section{Involutions and quandles}\label{sec:involutions}

We start with some necessary recollections and extensions to the profinite setting of the main algebraic structure that we are interested in: quandles, which describe symmetries in a way different from, but of course related to, groups. Suitable references are~\cite{Brieskorn, Fenn+Rourke, Joyce, Matveev}. The terminology and notation vary considerably among the various sources, and we shall stick to the following throughout:

\begin{definition}
Consider a pair~$(R,\rhd)$ consisting of a set~$R$ together with a binary operation~$\rhd$, a~\textit{magma}. A morphism of magmas is a map that is compatible with the operations~$\rhd$, and an automorphism is, of course, a bijective morphism. A~\textit{rack} is a magma such that, for all~$x$ in~$R$, the left-multiplication~$\lambda_x\colon y\mapsto x\rhd y$ is an automorphism of the underlying magma. The corresponding equation says
\begin{equation}\label{eq:rack_axiom}
x\rhd(y\rhd z) = (x\rhd y)\rhd(x\rhd z)
\end{equation}
for all~$x$,~$y$, and~$z$ in~$R$. A \textit{profinite rack} is a profinite space~$R$ together with a rack structure~\hbox{$\rhd\colon R\times R\to R$} that is continuous.
\end{definition}

Recall that a profinite space is a totally disconnected compact Hausdorff space, where a compact Hausdorff space is totally disconnected if we can separate any two points by subsets that are both open and closed at the same time.

\begin{remark}
The choice of definition of a profinite rack here is more debatable than in the case of profinite groups. Rubinsztein~\cite[Def.~2.1]{Rubinsztein}, for instance, defined topological racks as racks~$R$ together with a rack structure~\hbox{$\rhd\colon R\times R\to R$} that is continuous and such that all left-multiplication~$\lambda_x\colon R\to R$ are homeomorphisms. Of course, for a compact Hausdorff space, such as a profinite space, continuous bijections are homeomorphisms, so this requirement is automatic. In fact, more is true. As the rack structure~\hbox{$\rhd$} is continuous if and only its adjoint~$R\to\Homeo(R)$ to the topological group~$\Homeo(R)$ of self-homeomorphisms is continuous~(see~\cite[Thm.~5.1.8]{Laures--Szymik} for instance) we can use the continuity of the inversion in the topological group~$\Homeo(R)$, if~$R$ is compact Hausdorff, to deduce the continuity of the inverted operation~$\rhd^{-1}\colon R\times R\to R$ given by~$x\rhd^{-1}y=\lambda_x^{-1}(y)$. This recovers that our profinite racks satisfy Takahashi's more restrictive definition~\cite[Def.~2.7]{Takahashi} of a topological rack. Singh~\cite[Def.~3.6]{Singh} defines profinite racks as inverse limits of finite racks, and these examples are profinite racks in our sense. Overall, we find the definition chosen above to be just right for the context of the present paper.
\end{remark}

\begin{remark}
By definition, racks bring their own automorphisms, namely the left-multiplications, one for each element~$x$. Still, there is also a~\textit{natural} automorphism, namely~$x\mapsto x\rhd x$, which is compatible with all rack morphisms and which generates all natural automorphisms~\cite[Thm.~5.4]{Szymik:1}, and which does not necessarily have the form of a left-multiplication. 
\end{remark}

\begin{definition}
If all natural automorphisms of a rack are the identity, then the rack is called a~\textit{quandle}. In symbols, this property means
\begin{equation}
x\rhd x=x
\end{equation}
for all elements~$x$. A \textit{profinite quandle} is a profinite rack that is a quandle.
\end{definition}

\begin{examples} 
The underlying set of any group forms a quandle under conjugation, i.e., if we define~\hbox{$g\rhd x=gxg^{-1}$}. Any profinite group gives rise to a profinite quandle in this way.
\end{examples}

\begin{definition}
A rack is called~\textit{involutory} if all left-multiplications have order at most two. It is equivalent to asking that the equation
\begin{equation}
x\rhd(x\rhd y)=y
\end{equation}
holds for all~$x$ and~$y$. 
\end{definition}

\begin{remark}
An involutory quandle is sometimes called a {\em kei}, following Taka\-saki~\cite{Takasaki}. The term{\it~Takasaki quandle} is sometimes used in his honour. Our present terminology follows Joyce~\cite{Joyce}. Fenn and Rourke~\cite[Sec.~1.3]{Fenn+Rourke} call such structures{\it~involutive}, but according to Brieskorn~\cite[Sec.~2]{Brieskorn}, this term refers to a completely different property that we will not need here. We refer to~\cite[Sec.~10f.]{Joyce} for more information on involutory quandles.
\end{remark}

\begin{remark}
If~$G$ is a group, the formula~$g\rhd x=gx^{-1}g$ defines a binary operation that turns the whole underlying set of~$G$ into an involutory quandle. This class of examples seems to originate in Loos' work~\cite{Loos} on the structure of symmetric spaces. As a particular case, if~$G$ is the Euclidean plane, with the group structure given by addition, then we have~\hbox{$g\rhd x=2g-x$}, so that left-multiplication is reflection in the point~$g$. Figure~\ref{fig:point_reflections} illustrates the rack axiom. 
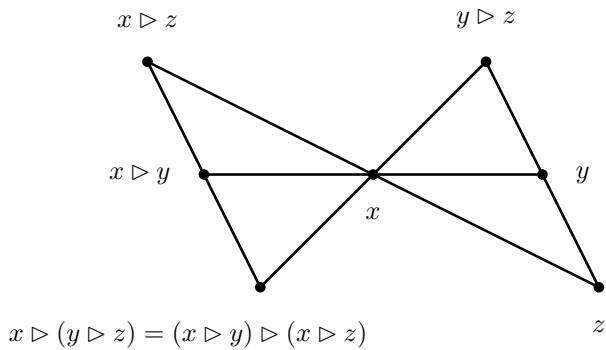
\begin{figure}
\begin{center}
\begin{tikzpicture}[scale=.75, line width=1pt, label distance=5pt]
\draw[black] (2,0)--(6,4);
\draw[black] (1,2)--(7,2);
\draw[black] (0,4)--(8,0);
\draw[black] (0,4)--(2,0);
\draw[black] (6,4)--(8,0);
\filldraw[black] (4,2) circle (2pt) node[label={270:{$x$}}] {};
\filldraw[black] (7,2) circle (2pt) node[label={0:{$y$}}] {};
\filldraw[black] (8,0) circle (2pt) node[label={270:{$z$}}] {};
\filldraw[black] (6,4) circle (2pt) node[label={90:{$y\rhd z$}}] {};
\filldraw[black] (0,4) circle (2pt) node[label={90:{$x\rhd z$}}] {};
\filldraw[black] (1,2) circle (2pt) node[label={180:{$x\rhd y$}}] {};
\filldraw[black] (2,0) circle (2pt) node[label={270:{$x\rhd (y\rhd z)=(x\rhd y)\rhd(x\rhd z)\phantom{nothing here}$}}] {};
\end{tikzpicture}
\end{center}
\caption{The rack axiom~\eqref{eq:rack_axiom} for symmetric spaces}
\label{fig:point_reflections}
\end{figure}
This particular rack is an involutory quandle. Bachmann~\cite{Bachmann:AGS,Bachmann:ESG}, building on earlier work of Hjelmslev, has used groups generated by involutions as a foundation for geometry in a way that is distinctly different from Coxeter's better-known work on reflection groups.
\end{remark}

\begin{examples} 
For any group~$G$, the subset
\begin{equation}\label{eq:Inv}
\Inv(G)=\{s\in G\mid s^2=e\not=s\}
\end{equation}
of~\textit{involutions} in~$G$ forms an involutory quandle with respect to conjugation. We emphasise that this convention excludes the identity from being an involution. If~$G$ is profinite, and~$\Inv(G)$ is a closed subset, then~$\Inv(G)$ will also be profinite. We shall see below, in Lemma~\ref{lem:closed}, that this is always the case in the situations we are interested in.
\end{examples}

To understand an algebraic structure, it is vital to understand its free models. In fact, any algebraic theory is determined by its free models~(see Lawvere~\cite{Lawvere}), but we will not use the language of algebraic theories here except for in Remark~\ref{rem:Lawvere} further down. However, we do need to discuss the free models in the relevant examples. For the theory of groups, for instance, we have the free groups~$\rmF(S)$ generated by sets~$S$, and if~$S$ has~$n$ elements, then the elements of~$\rmF(S)$ are in bijection with the natural operations on groups of arity~$n$. Specifically, if~\hbox{$S=\{x,y\}$}, then the multiplication corresponds to the element~\hbox{$xy\in\rmF\{x,y\}$}, whereas~$xyx^{-1}\in\rmF\{x,y\}$ gives conjugation.  For the theory of quandles, we can describe the free models as follows.

\begin{proposition}
The free quandle~$\FQ(S)$ on a set~$S$ is the union of the conjugacy classes of the basis elements~$S\subseteq\rmF(S)$ inside the free group~$\rmF(S)$; the operation~$\rhd$ is given by conjugation.
\end{proposition}

We refer to~\cite[Thm.~4.1]{Joyce} and~\cite[Satz~2.5]{Krueger} for proofs. 

We can also describe free involutory quandles explicitly in a similar manner. We will have to replace the free groups~$\rmF(S)$, which we can think of as the universal Artin groups, with the universal Coxeter groups. Recall that an Artin group can be presented in the form~$\langle S|R\rangle$, where~$S$ is a set of generators, and~$R$ is a set of relations of the form~\hbox{$sts...=tst...$} for~$s\not=t$ in~$S$. The~\textit{universal} Artin group corresponds to the case~$R=\emptyset$, i.e., just the free group generated by the set~$S$. Coxeter groups have the additional relations~$s^2=e$ for all generators~$s$ in~$S$, so that the~\textit{universal} Coxeter group~\cite{Muhlherr} on a set~$S$ of generators is the quotient
\[
\rmF_2(S)=\langle\, s\in S\mid s^2=e\text{ for all }s\in S\,\rangle
\]
of the free group~$\rmF(S)$ by the relations that say that the generating elements in~$S$ are involutions. When the set~$S$ has~$n$ elements, then the group~$\rmF_2(S)\cong(\bbZ/2)^{\star n}$ is the coproduct of~$n$ copies of the group of order~$2$. 

\begin{proposition}
The free involutory quandle~$\FQ_2(S)$ on a set~$S$ is the union of the conjugacy classes of the basis elements~$S\subseteq\rmF_2(S)$ inside the universal Coxeter group~$\rmF_2(S)$.
\end{proposition}

We refer to~\cite[Cor.~10.3]{Joyce} and~\cite[Satz~2.5]{Krueger} once more. 

\begin{remark}
We remark that the group~$\rmF_2(S)$ is not defined by universal identities or laws; the equation~$s^2=e$ is required for the generators~$s$ only. It does not necessarily hold for every group element. If we require~$g^2=e$ for all elements~$g$ in a group~$G$, the group~$G$ must be abelian. In fact, this law defines the variety of elementary abelian~$2$--groups, and~$\rmF_2(S)$ is not abelian unless~$S$ has at most one element. So, while the functor~$\rmF_2$ from the category of sets to the category of groups is a left-adjoint, it is not left-adjoint to the canonical forgetful functor. In contrast, the involutory quandle~$\FQ_2(S)$ is characterised as a left-adjoint to the forgetful functor from the category of involutory quandles; it is a free object among involutory quandles.
\end{remark}

\begin{proposition}\label{prop:FQ_2}
We have~$\FQ_2(S)=\{\,g\in\rmF_2(S)\mid g^2=e\not=g\,\}$. 
\end{proposition}

\begin{proof}
Clearly, every conjugate of a generator~$s\in S$ in~$\rmF_2(S)$ is an involution. The converse also holds: Every involution is conjugate to a generator. This is elementary to verify, and it is also a consequence of the Kurosh subgroup theorem; see~\cite[Cor.~4.1.4]{Magnus+Karrass+Solitar}, for instance.
\end{proof}

\begin{remark}\label{rem:Lawvere}
We can rephrase the observation leading to Proposition~\ref{prop:FQ_2} in more conceptual terms as follows. In any category with finite sums, we can take the full subcategory whose objects are the finite sums of any chosen object~$Z$, and this will be an algebraic theory in the sense of Lawvere~\cite{Lawvere}. For instance, if we consider the category of models for an algebraic theory and choose for~$Z$ a free model on one generator, we recover that algebraic theory in this way. In particular, for the group~\hbox{$Z=\bbZ$}, the free groups recover the algebraic theory of groups. More generally, we now see that any group~$Z$ defines an algebraic theory such that the sets~$\Mor(Z,G)$ of group morphisms from~$Z$ into groups~$G$ are models for that theory. For the case~\hbox{$Z=\bbZ/2$} of interest to us here, we see that the sets~\hbox{$\Mor(\bbZ/2,G)\cong\Inv(G)\cup\{e\}$} support models for an algebraic theory: the theory of involutory quandles with unit, where a unit~$u$ is an element that is both fixed and fixing:~$x\rhd u=u$ and~$u\rhd y =y$ for all~$x,y$~(see~\cite[Def.~2.8]{Lawson--Szymik} for this terminology). In fact, the theory defined by~$Z=\bbZ/2$ is the theory of involutory quandles with unit. This insight goes back to Bergman~(see his~ICM talk~\cite{Bergman:ICM} for general statements and his textbook~\cite[Sec.~10.11]{Bergman:book} for details).
\end{remark}

For the Galois theory context, we also need the profinite versions~$\hatF(S)$ and~$\hatF_2(S)$ of the groups~$\rmF(S)$ and~$\rmF_2(S)$, respectively, where~$S$ can now be any profinite space. If the space~$S$ is finite, we can build the profinite groups~$\hatF(S)$ as the profinite completions of the free groups~$\rmF(S)$, namely the limit~$\lim_N\rmF(S)/N$, where~$N$ ranges over the finite index normal subgroups of~$\rmF(S)$. If~$S$ is, more generally, profinite, we have two choices: On the one hand, we can set~\hbox{$\hatF(S)=\lim_N\rmF(S)/N$} as before, but require that~\hbox{$gN\cap S$} is open in~$S$ for the (finitely many) cosets~$gN$ of~$N$ in~$\rmF(S)$. On the other hand, we can set~$\hatF(S)=\lim_R\hatF(S/R)$, where~$R$ ranges over the equivalence relations on~$S$ that have a finite discrete quotient. Both constructions give isomorphic results~\cite[Prop.~1.7]{Gildenhuys--Lim}. This and other suitable references come from Gildenhuys~et~al.~\cite{Gildenhuys--Lim, Gildenhuys--Ribes}, which build on earlier work of Neukirch~et~al.~\cite{Neukirch:1,Neukirch:2,Binz+Neukirch+Wenzel}, Mac Lane~\cite{Mac_Lane}, and the writings of Haran and Jarden~\cite{Haran--Jarden:Pisa,Haran--Jarden:JPAA}. The last two are especially relevant because of the next result, Proposition~\ref{prop:profinite_Kurosh}, which they provide. They concern the groups~$\hatF_2(S)$ that are characterised by the property that a continuous homomorphism~$\hatF_2(S)\to G$ into a profinite group~$G$ is the same thing as a continuous map~$S\to G$ such that the elements in the image square to the identity element~$e$. To construct them, we can follow~\hbox{\cite[\S1 and Prop.~3.1]{Haran--Jarden:JPAA}} and use~$\hatF_2(S)=\lim_N\rmF(S)/N$ with the added requirement that~$N$ contains all squares of elements in~$S$ or work with~$\hatF_2(S)=\lim_R\hatF_2(S/R)$. There are profinite versions of the Kurosh subgroup theorem:

\begin{proposition}\label{prop:profinite_Kurosh}
Every involution in~$\hatF_2(S)$ is conjugate to a unique element in~$S$, and the centraliser of any involution consists of only that involution and the identity.
\end{proposition}

\begin{proof}
If~$S$ is a finite set of~$n$ elements, Haran--Jarden~\cite[Prop.~6.1]{Haran--Jarden:Pisa} show that there are exactly~$n$ conjugacy classes of involutions in~$\hatF_2(S)$, represented by the elements in~$S$. In~\cite[Cor.~3.2]{Haran--Jarden:JPAA}, they extend this to the general case.
\end{proof}

\begin{proposition}\label{prop:universal}
For every profinite space~$S$, there is a profinite involutory quandle~$\hatFQ_2(S)$ together with a continuous map~$S\to\hatFQ_2(S)$ that satisfies the universal property: Restriction gives a bijection between the set of continuous homomorphism~$\hatFQ_2(S)\to Q$ into a profinite involutory quandle~$Q$ and the set of continuous maps~\hbox{$S\to Q$}. Furthermore, there is a continuous map~$\hatFQ_2(S)\to S$ that identifies the target~$S$ with the orbit space of the~$\hatF_2(S)$--action on~$\hatFQ_2(S)$ via conjugation.
\end{proposition}

\begin{proof}
Let us set
\begin{equation}\label{eq:FQ=Inv(F)}
\hatFQ_2(S)=\Inv(\hatF_2(S)).
\end{equation}
This is a profinite quandle: the involutions form a closed subset of the topological space~$\hatF_2(S)$ as they are the image of the continuous map~\hbox{$\hatF_2(S)\times S\to \hatF_2(S)$} that sends~$(g,s)$ to~$gsg^{-1}$. The involutory quandle~$\hatFQ_2(S)$ comes with a natural continuous map from~$S$ because the image of the canonical map~$S\to\hatF_2(S)$ lies inside~$\hatFQ_2(S)$. So, any continuous morphism~$\hatFQ_2(S)\to Q$ from~$\hatFQ_2(S)$ into a profinite involutory quandle~$Q$ restricts to a continuous map~$S\to Q$. 

Conversely, if~\hbox{$\phi\colon S\to Q$} is any continuous map, we can extend it functorially to a continuous morphism~\hbox{$\hatF_2(\phi)\colon\hatF_2(S)\to\hatF_2(Q)$} of groups, and this restricts to a continuous morphism~\hbox{$\hatFQ_2(\phi)\colon\hatFQ_2(S)\to\hatFQ_2(Q)$} of involutory quandles. We get a continuous morphism~$\hatFQ_2(S)\to Q$ by composition of this morphism~$\hatFQ_2(\phi)$ with the continuous orbit map~$\hatFQ_2(Q)\to Q$, where~$Q$ is identified with the space of orbits in the subspace~$\hatFQ_2(Q)=\Inv(\hatF_2(Q))$ for the continuous action of the group~$\hatF_2(Q)$ on itself via conjugation. The preceding two constructions are inverse to each other.

The final statement is immediate from what we have already used in the previous paragraph of the proof: The composition~\hbox{$S\longrightarrow\hatFQ_2(S)\longrightarrow\hatFQ_2(S)/\hatF_2(S)$}
is a homeomorphism, and we can use its inverse to obtain the desired map.
\end{proof}

\begin{remark}\label{rem:old_4.2}
We remark that any free quandle generates a free group: Ignoring topologies for a moment, the~\textit{enveloping group}~$\Gr(Q)$ of a quandle~$Q$ is, by definition, the value of the left-adjoint~$\Gr$ to the forgetful functor from the category of groups to the category of quandles. To construct it, up to isomorphism, we can start with the free group~$\rmF(Q)$ on the set~$Q$, with~$f(q)$ the basis element corresponding to an element~\hbox{$q\in Q$}, and pass to the quotient by the relations
\begin{equation}\label{eq:relations}
f(x\rhd y)=f(x)f(y)f(x)^{-1}.
\end{equation}
As a left adjoint, the functor~$\Gr$ preserves colimits, and this shows that we have an isomorphism~\hbox{$\Gr(\FQ(S))\cong\rmF(S)$}. The construction can be adapted to profinite quandles~$Q$. Ignoring the quandle structure on~$Q$ for a moment, the profinite space~$Q$ freely generates the profinite group~$\hatF(Q)$, and~$\hatGr(Q)$ is the quotient of it by the closed normal subgroup generated by the relations~\eqref{eq:relations}. Similarly, the enveloping group of the quandle~$\hatFQ_2(S)$ is canonically isomorphic to the group~$\hatF_2(S)$ for formal reasons; both sides satisfy the same universal property. All that is to point out that, in contrast, the proof of Theorem~\ref{thm:AS(QQ)} will go the other way. We will argue for the freeness of a quandle from the freeness of its enveloping group, and such an implication is not just a categorical formality. Any argument to deduce the freeness of a quandle~$Q$ from the freeness of its enveloping group has to be non-formal.
\end{remark}

\begin{remark}\label{rem:old_4.3}
We shall also use the fact that the map~$\hatFQ_2(S)\to\hatF_2(S)$ is injective. This is not automatic, either. There are quandles~$Q$ for which the canonical map from~$Q$ into the enveloping group is not injective. For example, if~$Q=\{x,y,z\}$ is the involutory quandle with~$\lambda_x$ transposing~$y$ and~$z$, and~$\lambda_y=\lambda_z=\id_Q$, then~$y\rhd x=x$ implies that~$x$ commutes with~$y$ in the enveloping group, and~$x\rhd y=z$ implies that~$x$ conjugates~$y$ into~$z$ in the enveloping group. Therefore, the elements~$y$ and~$z$ must become equal in the enveloping group.
\end{remark}


\section{Artin--Schreier quandles}\label{sec:AS(F)}

With this section, we turn our attention to fields. We start by establishing notation and context. Suitable references for this material are Lam~\cite{Lam} and Knebusch--Scheiderer~\cite{Knebusch--Scheiderer}.

We recall that an~\textit{ordering} of a field~$F$ is given by a subset~$P\subset F$ of elements that is closed under addition and multiplication and satisfies~$P\cap-P=\{0\}$ and~\hbox{$P\cup-P=F$}, where~\hbox{$-P=\{a\in F\mid -a\in P\}$}. It follows that all sums of squares are in~$P$. A field is~\textit{formally real} if and only if we can order it. A~\textit{Euclidean field} is a field for which the set of squares is an ordering. In a Euclidean field, the ordering is necessarily unique. A field is~\textit{real closed} if, in addition, each polynomial of odd degree with coefficients in~$F$ has at least one root in~$F$.

The~\textit{real spectrum} of a field~$F$ is the set~$\Ord(F)$ of orderings of~$F$ equipped with the~\textit{Harrison topology}: given any element~$a\not=0$ in~$F$, the set of orderings for which~$a$ is positive is a generating open set. With this topology, the real spectrum is a profinite space. 

The~\textit{Grothendieck--Witt ring}~$\GW(F)$ of a field~$F$ (of characteristic not~$2$) is the group completion of the semi-ring of non-degenerate quadratic forms. The dimension gives a well-defined homomorphism~\hbox{$\GW(F)\to\bbZ$}. 
The~\textit{Witt ring}~$\rmW(F)$ is the quotient of the Grothendieck--Witt ring~$\GW(F)$ by the ideal generated by the hyperbolic plane. The parity of the dimension gives a well-defined homomorphism~$\rmW(F)\to\bbZ/2$. 

\begin{proposition}\label{prop:Burnside}
The real spectrum of a field~$F$ is in natural bijection with the set of ring morphisms~\hbox{$\rmW(F)\to\bbZ$} from the Witt ring of~$F$ to the ring of integers.
\end{proposition}

\begin{proof}[Sketch of proof]
Given any ordering~$P$ of~$F$, the associated morphism sends a quadratic form to its~$P$--signature---the difference of the numbers of positive and negative diagonal entries, where the meaning of~`positive' and~`negative' is determined by the ordering~$P$. Another way to think about this: an ordering determines an embedding into a real closed field so that we can use functoriality and the fact that the Witt ring of a real closed field is canonically isomorphic to~$\bbZ$ through the signature. We refer to~\cite[Thm.~I.2.3]{Knebusch--Scheiderer} for more detail.
\end{proof}

\begin{remark}
The kernel~$I$ of the parity homomorphism~$\rmW(F)\to\bbZ/2$ defines the~$I$--adic filtration on the Witt ring~$\rmW(F)$, and the associated graded ring is then~\hbox{$\gr\rmW(F)=\rmW(F)/I\oplus I/I^2\oplus I^2/I^3\oplus\dots$}. It is a fact that~$\gr\rmW(F)$ is a graded commutative algebra over~$\bbF_2$, and it is isomorphic to mod~$2$ Milnor~K-theory and mod~$2$ Galois cohomology via Milnor's conjecture, which is now a theorem~\cite{OVV}.
Therefore, we can recover~$\gr\rmW(F)$ from the absolute Galois group of~$F$. Can we recover~$\rmW(F)$ itself from a refinement of the absolute Galois group? This was a non-trivial question because there are examples of fields with isomorphic~$\gr\rmW(F)$'s but non-isomorphic~$\rmW(F)$'s~(see~\cite{Pfister:DMV}); it has been answered in the affirmative by Min\'a\v c and Spira in~\cite{Minac+Spira:1,Minac+Spira:2}. The authors show that the Witt ring determines and is determined~(modulo a small restriction) by a specific quotient of the absolute Galois group~(see~\cite[Thm.~3.8]{Minac+Spira:2} in particular).
\end{remark}

\begin{remark}
Half a century ago, Dress~\cite{Dress} found a relation between Witt rings and Burnside rings: 
If~$E|F$ is a Galois extension with Galois group~$G$, then there is a map~\hbox{$\rmA(G)\to\rmW(F)$} from the Burnside ring~$\rmA(G)$ of the Galois group to the Witt ring of the ground field that sends~$G/H$ to~$(E^H,\tr)$, where~$\tr$ is the usual trace form~\hbox{$(a,b)\mapsto\tr_{E|F}(ab)$}. 
\end{remark}

\begin{proposition}
There is an identification between the real spectrum of a field~$F$ and the set of conjugacy classes of involutions in its absolute Galois group~$\Gal(F)$.
\end{proposition}

\begin{proof}[Sketch of proof]
Given a conjugacy class~$[\sigma]$ of an involution~$\sigma$ in~$\Gal(F)$, we have the ordering~\hbox{$P[\sigma]=\{a\in F\mid \sigma(\sqrt{a})=\sqrt{a}\}$}. Note that the condition~$\sigma(\sqrt{a})=\sqrt{a}$ does not depend on the choice of the square root~$\sqrt{a}$ of~$a$. We refer to~\cite[Thm.~I.11.3]{Knebusch--Scheiderer} for more detail.
\end{proof}

We take this observation as an incentive to turn our attention towards involutions.

Let~$F$ be a field and~$\Gal(F)$ be its absolute Galois group. From the work of Artin and Schreier cited in the introduction, see~\cite[Satz~4]{AS:2} in particular, it is known that the group~$\Gal(F)$, even though profinite, can only have non-trivial elements of finite order if~\hbox{$\ch(F)=0$}, and all of these elements necessarily have order~$2$, i.e., they are involutions. We refer to~\cite{Jarden--Videla} for an interesting discussion of torsion in absolute Galois groups from a different angle. The following lemma collects useful group-theoretic statements about involutions in absolute Galois groups.

\begin{lemma}\label{lem:centraliser}
Let~$F$ be a field and let~$\sigma\in\Gal(F)$ be an involution with fixed field~$\Fix(\sigma)$. Then~$\Aut(\Fix(\sigma)|F)=\{\id\}$. The centraliser of the involution~$\sigma$ within~$\Gal(F)$ is the subgroup~$\langle\sigma\rangle$ generated by~$\sigma$. The normaliser of~$\Gal(\Fix(\sigma))$ within~$\Gal(F)$ is~$\Gal(\Fix(\sigma))$ itself.
\end{lemma}

\begin{proof}
We first observe that~$\Fix(\sigma)$ is real closed and, therefore, has a unique ordering consisting of its squares. This means every field automorphism of~$\Fix(\sigma)$ preserves the ordering. 

If~$\tau$ is a field automorphism of~$\Fix(\sigma)$ that fixes~$F$, and~$a\in\Fix(\sigma)$ with~$\tau(a)\not=a$, we have~$a<\tau(a)$, say. Then~$a<\tau(a)<\tau^2(a)<\tau^3(a)<\dots$, contradicting the algebraicity of~$a$: the~$\tau$--orbit of~$a$ must be finite. The argument for~$a>\tau(a)$ is similar, and we deduce~$\tau(a)=a$ for all~$a\in\Fix(\sigma)$. This shows~\hbox{$\Aut(\Fix(\sigma)|F)=\{\id\}$}.

If an element~$\tau\in\Gal(F)$ lies in the centraliser of~$\sigma$, then~$\tau$ sends~$\Fix(\sigma)$ into itself. From what we have shown before, it follows that~$\tau$ is the identity on~$\Fix(\sigma)$ and lies in~\hbox{$\Gal(\Fix(\sigma))=\langle\sigma\rangle$}. The last statement is simply yet another reformulation of the preceding statement.
\end{proof}

Later, we shall need the following stronger statement about involutions in absolute Galois groups.

\begin{lemma}\label{lem:finite_subgroups}
A finite subgroup of an absolute Galois group has order at most~$2$. 
\end{lemma}

\begin{proof}
By Artin--Schreier, we already know that every element of the finite subgroup~$G$ must have order~$1$ or~$2$. It follows from group theory that~$G$ is an elementary abelian~$2$--group, that is~$G\cong(\bbZ/2)^n$ for some~$n$. We show~$n=1$ by showing that~$\Gal(F)$ cannot contain a subgroup of the form~$\bbZ/2\times\bbZ/2$: Given an involution~\hbox{$\sigma\not=\id$}, Lemma~\ref{lem:centraliser} says that its centraliser within~$\Gal(F)$ is the subgroup~$\langle\sigma\rangle$ generated by~$\sigma$. Therefore, there can be no other element in the group~$\Gal(F)$ that commutes with~$\sigma$.
\end{proof}

With these preliminaries out of our way, we can now define a new invariant of fields~$F$. Recall that we have defined, in~\eqref{eq:Inv}, the subset~$\Inv(G)$ of involutions of a group~$G$.

\begin{lemma}\label{lem:closed}
For any field~$F$, the set~$\Inv(\Gal(F))$ is a closed subspace of the absolute Galois group~$\Gal(F)$. Therefore, it is profinite as well.
\end{lemma}

\begin{proof}
The subset~$\Inv(\Gal(F))\cup\{\id\}$ is closed, as this subset is defined by the equation~\hbox{$\sigma^2=\id$}. It remains for us to show that the identity is an isolated point. To see that, we consider two cases. First, if~$F(\sqrt{-1})=F$, then~$\Gal(F)$ contains no involution, and we are talking about the empty set. Else, the field~$F(\sqrt{-1})$ is a quadratic extension, so that~$U=\Gal(F(\sqrt{-1}))$ is an index 2 subgroup of~$\Gal(F)$, hence open and containing the identity. But~$F(\sqrt{-1})$ is not formally real, so that~$U$ contains no involution.
\end{proof}

\begin{definition}
Let~$F$ be a field. The~\textit{Artin--Schreier quandle}~$\AS(F)$ of~$F$ is the profinite involutory quandle
\begin{equation}
\AS(F)=\Inv(\Gal(F)).
\end{equation}
It consists of the set of involutions inside the absolute Galois group~$\Gal(F)$ of~$F$ together with the operation~$\rhd$ given by conjugation and its topology inherited from~$\Gal(F)$.
\end{definition}

\begin{remark}
The Artin--Schreier quandle~$\AS(F)$ of a field~$F$, just as its absolute Galois group~$\Gal(F)$, depends on the choice of an algebraic closure of~$F$. If~\hbox{$\omega\colon\Omega_1\to\Omega_2$} is an~$F$--isomorphism of algebraic closures of~$F$, then~$\sigma\mapsto \omega\sigma\omega^{-1}$ is an isomorphism~$\Aut(\Omega_1|F)\to\Aut(\Omega_2|F)$ and this isomorphism restricts to an isomorphism between the Artin--Schreier quandles computed using~$\Omega_1$ and~$\Omega_2$, respectively. As usual, we will suppress the choice of an algebraic closure from the notation.
\end{remark}

\begin{remark}
It is usually a bad idea to forget everything from a group except its conjugation. In~\cite{Szymik--Vik}, the authors argue that it is more reasonable to keep at least the group operations of arity~$0$ and~$1$ as well, which are the unit~$e$ and the power operations~$g\mapsto g^n$ for all integers~$n$. In the present situation, it is easy to recover this more refined structure from the knowledge of only the operation~$\rhd$ on~$\AS(F)$: By our definition of an involution, the unit~$e=\id$ of~$\Gal(F)$ is disjoint from~$\AS(F)$, and the power operations on the disjoint union~\hbox{$\AS(F)\cup\{\id\}$} are determined by~$\sigma^n=\id$ if~$n$ is even, as all elements~$g$ considered have order at most~$2$, and~$\sigma^n=\sigma$ if~$n$ is odd, for the same reason. This suggests that we are not missing out on something obvious when considering the set~$\AS(F)$ together with only the conjugation~$\rhd$ on it as an algebraic invariant of~$F$.
\end{remark}

Before we come to examples, let us note the following algebraic result about a particular property of Artin--Schreier quandles.

\begin{proposition}\label{prop:anabelian}
Let~$\AS(F)$ be the Artin--Schreier quandle of a field~$F$. For~$\sigma$ and~$\tau$ in~$\AS(F)$ we have~$\sigma\rhd\tau=\tau$ if and only if~$\sigma=\tau$.
\end{proposition}

\begin{proof}
The~`if' direction is part of the quandle axioms, which are evident by now. For the~`only if' direction, note that~$\sigma\rhd\tau=\tau$ is equivalent to~$\sigma\tau=\tau\sigma$, so that~$\tau$ must be in the centraliser of~$\sigma$. It cannot be the identity, by definition of~$\AS(F)$, so the result follows from Lemma~\ref{lem:centraliser}.
\end{proof}

Unlike groups, which always have a neutral element, quandles can be empty. The following result characterises when this happens for Artin--Schreier quandles.

\begin{proposition}\label{prop:formally_real}
Let~$F$ be a field. Then~$\AS(F)$ is non-empty if and only if~$F$ is formally real. This can only happen if~\hbox{$\ch(F)=0$}.
\end{proposition}

\begin{proof}
A field~$F$ is known to be formally real if and only if its absolute Galois group~$\Gal(F)$ contains an involution~(see~\cite[Kor.~I.6.2]{Knebusch--Scheiderer}, for instance). 
\end{proof}

\begin{example}
The field~$\bbC$ of complex numbers is not formally real, and the Artin--Schreier quandle~$\AS(\bbC)$ is empty. The same holds for all finite fields and the cyclotomic fields~$\bbQ(\zeta_n)$ as soon as we are in the non-trivial range~$n\geqslant 3$.
\end{example}

Proposition~\ref{prop:formally_real} implies that it is not interesting---for our purposes---to consider fields that are not formally real. Therefore, unless otherwise stated, we can assume we are dealing with formally real fields.

\begin{proposition}\label{prop:real closed}
Let~$F$ be a field. Then~$\AS(F)$ is a singleton if and only if~$F$ is real closed.
\end{proposition}

\begin{proof}
We know from the work of Artin and Schreier~(see again~\cite[Satz~4]{AS:2}) that a field~$F$ is real closed if and only if its absolute Galois group has order~$2$.
\end{proof}

\begin{example}
The field~$\bbR$ of real numbers is real closed, and the unique element in the involutory quandle~$\AS(\bbR)=\{\gamma\}$ is complex conjugation~$\gamma\colon z\mapsto\overline{z}$ inside the absolute Galois group~\hbox{$\Gal(\bbR)=\Aut(\bbC|\bbR)=\{\id,\gamma\}$} of order~$2$.
\end{example}

\begin{remark}
A bijection exists between the set~$\AS(F)$ and the set of real closed subfields containing~$F$ of the algebraic closure of~$F$ used to define the absolute Galois group. An involution~$\sigma$ corresponds to its fixed field~$\Fix(\sigma)$.
\end{remark}

The following result shows that the examples characterised in Propositions~\ref{prop:formally_real} and~\ref{prop:real closed} are the only ones where the Artin--Schreier quandle is finite.

\begin{proposition}\label{prop:infinite}
If~$\AS(F)$ contains more than one element, it is infinite.
\end{proposition}

\begin{proof}
When the Artin--Schreier quandle~$\AS(F)$ contains more than one element, the absolute Galois group~$\Gal(F)$ contains two different involutions, say~\hbox{$\sigma\not=\tau$}. Consider the morphism~$\rmF_2\{s,t\}\to\Gal(F)$ of groups that sends the generator~$s$ to~$\sigma$ and~$t$ to~$\tau$, and suppose there was a non-trivial kernel.
Recall that the group~$\rmF_2\{s,t\}$ is the infinite dihedral group, which contains an infinite cyclic subgroup generated by the product~$st$ of the generators, with inverse~$ts$, and all other elements are alternating products~$ststs\dots$ and~$tstst\dots$ of odd length, which are conjugate to their~`middle' factor, either~$s$ or~$t$. The kernel of the morphism above cannot contain~$s$ or~$t$ by definition, and it cannot contain any other involution, as these are each conjugate to either~$s$ or~$t$ inside~$\rmF_2\{s,t\}$. It follows that the kernel is a subgroup of the infinite cyclic subgroup generated by~$st$, so that the image is a finite~(dihedral) subgroup of~$\Gal(F)$, and it contains two different involutions~\hbox{$\sigma\not=\tau$}. But the finite subgroups of~$\Gal(F)$ have order at most~$2$ by Lemma~\ref{lem:finite_subgroups}---a contradiction. Hence, the morphism~$\rmF_2\{s,t\}\to\Gal(F)$ is injective, and~$\Gal(F)$ contains infinitely many involutions because this already holds for the infinite dihedral group~$\rmF_2\{s,t\}$, and these lie in~$\AS(F)$.
\end{proof}


\section{Number fields}\label{sec:QQ}

The prime field~$F=\bbQ$ of rational numbers is formally real but not real closed. It has a unique ordering. In other words, the real spectrum~$\Ord(\bbQ)$ is a singleton, and all involutions in the absolute Galois group~$\Gal(\bbQ)$ form one conjugacy class. If we choose any involution~\hbox{$\sigma\in\AS(\bbQ)$}, then conjugation induces a homeo\-morphism~\hbox{$\Gal(\bbQ)/\langle\sigma\rangle\cong\AS(\bbQ)$}, via~$g\mapsto g\sigma g^{-1}$, where the subgroup~$\langle\sigma\rangle$ is the centraliser of the involution~$\sigma$~(see Lemma~\ref{lem:centraliser}). In particular, there are uncountably many involutions inside~$\Gal(\bbQ)$, and the topology is reasonably well understood. In contrast, the algebraic structure of the Artin--Schreier quandle~$\AS(\bbQ)$ supported on it is much more interesting:

\begin{theorem}\label{thm:AS(QQ)}
The Artin--Schreier quandle~$\AS(\bbQ)$ of the rational number field~$\bbQ$ is a free profinite involutory quandle. A basis is given by a Cantor space of involutions inside~$\Gal(\bbQ)$.
\end{theorem}

This result indicates that the Artin--Schreier quandle~$\AS(\bbQ)$ is algebraically as complicated as possible, given the constraints of being an involutory quandle. Unfortunately, the proof will not display an explicit Cantor space inside~$\Gal(\bbQ)$ that is a basis for it. Instead, it implicitly uses Brouwer's characterisation~\cite{Brouwer} of the Cantor space: Every profinite space that meets the minimum requirements (i.e., it has to be non-empty, second countable, and does not have an isolated point) is homeomorphic to it.

\begin{proof}[Proof of Theorem~\ref{thm:AS(QQ)}]
Our goal is to show that there is a subspace~\hbox{$S\subseteq\Gal(\bbQ)$} homeo\-morphic to Cantor space and consisting of involutions so that the canonical extension
\begin{equation}\label{eq:canonical_quandles}
\hatFQ_2(S)\longrightarrow\AS(\bbQ)
\end{equation}
of the inclusion is an isomorphism of profinite quandles. As it is always a continuous morphism of profinite quandles, it will suffice to show that it is bijective.

The subset~$\AS(\bbQ)$ of involutions generates a normal subgroup of~$\Gal(\bbQ)$. Its closure~$N$ is a closed normal subgroup with fixed field an infinite Galois extension~$\Qtr$ of~$\bbQ$, the field of totally real numbers, and the group~$N$ can be identified with the absolute Galois group~$\Gal(\Qtr)$ of~$\Qtr$. Fried--Haran--V\"olklein~\cite{FHV93,FHV94} and~(independently) Pop~\cite{Pop} have shown that this absolute Galois group~$\Gal(\Qtr)$ of the field~$\Qtr$ is isomorphic to the free profinite Coxeter group~$\hatF_2(S)$ generated by a Cantor space~$S$. In other words, we can choose a subspace~$S$ as above so that the canonical morphism
\begin{equation}\label{eq:canonical_groups}
\hatF_2(S)\longrightarrow\Gal(\Qtr)
\end{equation}
is an isomorphism of profinite groups. Restriction to the subset of involutions gives an isomorphism
\[
\Inv(\hatF_2(S))\longrightarrow\Inv(\Gal(\Qtr))
\]
of profinite involutory quandles. The left hand side is~$\hatFQ_2(S)$ by definition~\eqref{eq:FQ=Inv(F)}. The right hand side is~$\Inv(\Gal(\bbQ))$ by our choice of~$\Gal(\Qtr)$ inside~$\Gal(\bbQ)$: it contains all involutions. The result then follows from the definition~$\AS(\bbQ)=\Inv(\Gal(\bbQ))$ of the Artin--Schreier quandle.
\end{proof}

\begin{remark}
For any field~$F$ of characteristic zero, passage to conjugacy classes gives a surjection~$\AS(F)\to\Ord(F)$ of profinite spaces from the Artin--Schreier quandle onto the real spectrum. This surjection should{\it~not} suggest, however, that we can compute the real spectrum~$\Ord(F)$ from the Artin--Schreier quandle~$\AS(F)$ alone. In particular, the reader may have wondered if we can recover the real spectrum of a field from its Artin--Schreier quandle as the space of orbits. To make perfect sense of this question, one would first have to exercise some care to define spaces of orbits for all profinite quandles. Any reasonable definition, however, would yield the following for Artin--Schreier quandles: the group~$\Gal(F)$ acts on the space~$\AS(F)$ via conjugation, and~$\Orb(\AS(F))$ is the space of orbits for the restriction of this action to the closed subgroup generated by~$\AS(F)$. Yet, we still cannot recover the real spectrum of a field from its Artin--Schreier quandle as the space of orbits in this sense~(see the following Example~\ref{ex:sqrt2}). Instead, we have two surjections.
\begin{equation}
\AS(F)\longrightarrow\Orb(\AS(F))\longrightarrow\Ord(F),
\end{equation}
and neither of them has to be injective. The Artin--Schreier quandle~$\AS(F)$ is contained inside~$\Gal(F)$, of course, and we can identify the real spectrum~$\Ord(F)$ with the orbits of the involutions under conjugation with~$\Gal(F)$. In contrast, the elements in~$\Orb(\AS(F))$ are the orbits under conjugation with respect to the closed subgroup generated by the involutions.
\end{remark}

\begin{example}\label{ex:sqrt2}
We give an example showing that we cannot recover the real spectrum of a field from its Artin--Schreier quandle as the set of orbits. First, note that we have an equality~\hbox{$\AS(\bbQ(\sqrt{2}))=\AS(\bbQ)$} of involutory quandles: every involution in~$\Gal(\bbQ)$ lies in the normal subgroup~$\Gal(\bbQ(\sqrt{2}))$~(see Theorem~\ref{thm:AS(F)} below for a more general statement.) Therefore, the Artin--Schreier quandles of~$\bbQ(\sqrt{2})$ and~$\bbQ$ are free on a Cantor space, and the space of orbits is the Cantor space in both cases. However, the rational number field~$\bbQ$ has a unique ordering, whereas the real quadratic extension~$\bbQ(\sqrt{2})$ has two. Some involutions in~$\Gal(\bbQ(\sqrt{2}))$ that are not conjugate in~$\Gal(\bbQ(\sqrt{2}))$ become conjugate in~$\Gal(\bbQ)$. 
\end{example}

Let us finally look at other examples of `small' fields. The other prime fields, finite fields, and all fields of positive characteristic are uninteresting to us, as their absolute Galois groups contain no involutions. We have already discussed~$\bbR$ and~$\bbC$, and the other local fields of characteristic~$0$~(i.e., the~$p$--adic fields~$\bbQ_p$ and their finite extensions) contain no involutions either: each~$\bbQ_p$ contains a square root of~$1-p^n$ for sufficiently large~$n$. This leaves us with the problem of extending Theorem~\ref{thm:AS(QQ)} to other number fields.

We first establish some more notation. Let~$F$ be a number field. We write~$d=\dim_\bbQ(F)$ for its degree. Then there are isomorphisms~\hbox{$\bbC\otimes_\bbQ F\cong\bbC^d$} and~$\bbR\otimes_\bbQ F\cong\bbR^{r_1}\oplus\bbC^{r_2}$ with~\hbox{$d=r_1+2r_2$}. The~$r_1$--element set~$\Mor(F,\bbR)$ of real embeddings~$F\to\bbR$, which we can identify with its real spectrum, sits inside the~$d$--element set~$\Mor(F,\bbC)$ of complex embeddings as the fixed points under the action of the Galois group~$\Aut(\bbC|\bbR)=\bbZ/2$ by complex conjugation. The orbits are the Archimedean places of~$F$, with~$r_1$ real ones and~$r_2$ complex ones. The number field~$F$ is totally real when~$d=r_1$ and~$r_2=0$, and~$F$ is totally imaginary when~$r_1=0$ and~$d=2r_2$. If~$F$ is Galois over~$\bbQ$, then it has to be either of those two, as the Galois group acts transitively on the set of Archimedean places of~$F$, showing that they all have to be of the same type. However, there are totally real number fields that are not Galois over~$\bbQ$.

\begin{theorem}\label{thm:AS(F)}
For any totally real number field~$F$, the Artin--Schreier quandle~$\AS(F)$ is a free profinite involutory quandle with a basis given by a Cantor space of involutions.
\end{theorem}

\begin{proof}
In the proof of Theorem~\ref{thm:AS(QQ)}, we have used the field~$\Qtr$ of totally real numbers. If the number field~$F$ is totally real, then it is contained in the field~$\Qtr$, and~$\Gal(F)$ sits in between~$\Gal(\Qtr)$ and~$\Gal(\bbQ)$. In particular, these groups all have the same subsets of involutions, and we can deduce~$\AS(F)=\AS(\bbQ)$. 
\end{proof}

In particular, if~$F$ is Galois over~$\bbQ$, then the Artin--Schreier quandle~$\AS(F)$ is either empty or isomorphic to~$\AS(\bbQ)$ as described above. 

One could argue that the totally real case is the most important. For instance, as Serre has remarked concerning the inverse Galois theory problem: All finite groups occur as Galois groups over~$\bbQ$ if and only if all finite groups occur as Galois groups of totally real extensions of~$\bbQ$~(see~\cite[Prop.~1]{Klüners--Malle}). However, there is something we can say concerning formally real number fields that are not totally real.

\begin{proposition}
Let~$F_0$ be a number field that is formally real, but not totally real. If~$F$ is any other formally real number field that is not totally real, then we have~\hbox{$\AS(F)\cong\AS(F_0)$}. In words, the Artin--Schreier quandle is independent of the choice.
\end{proposition}

\begin{proof}
We can consider the compositum~$E$ of~$\Qtr$ with~$F$, and this field will be a proper finite extension of~$\Qtr$. This larger field is even better than~$\Qtr$ in some sense. While the field~$\Qtr$ is pseudo real closed~(PRC)~by Pop~(see~\cite{Haran--Jarden--Pop}), but~\textit{not} Hilbertian, the proper finite extensions~$E$ of~$\Qtr$ are PRC~by Prestel~\cite{Prestel} and Hilbertian. In fact, every proper finite extension of~$\Qtr$ is Hilbertian~(see~\cite[Satz~9.7]{Weissauer} for a nonstandard proof and~\cite[Ch.~12]{Fried--Jarden} for a standard proof.) Hence, their absolute Galois groups
\[
\Gal(E)=\Gal(F)\cap\Gal(\Qtr)
\]
are known by~\cite[Cor.~2]{FV:PRC+Hilbert}: an argument due to Jarden shows that, in the formally real case, the profinite group~$\Gal(E)$ is abstractly isomorphic to the free product of~\hbox{$\Gal(\Qtr)\cong\hatF_2(S)$} with a free profinite group of countable infinite rank~(see~\cite[Cor.~1]{FV:PRC+Hilbert} and the remark following it). In particular, we see that the isomorphism type of~$\Gal(E)$, while different from~$\Gal(\Qtr)$, is independent of~$F$. Computing the set of involutions in~$\Gal(F)$ is the same as computing the involutions in the intersection of~$\Gal(F)$ with the subgroup generated by the involutions:
\[
\AS(F)=\Inv(\Gal(F))=\Inv(\Gal(F)\cap\Gal(\Qtr)).
\]
From the discussion above, we know that~$\Inv(\Gal(F)\cap\Gal(\Qtr))=\Inv(\Gal(E))$ is independent of~$F$. 
\end{proof}

Therefore, to determine the Artin--Schreier quandle of a field~$F$ as above, it would be sufficient to work out one particular example. However, this is difficult for reasons related to the splitting behaviour of involutions across subgroups~(see the following example). In addition, it is not known whether subquandles of free involutory quandles are free.

\begin{example}
The absolute Galois group~$\Gal(F)$ of a number field~$F$ is a finite index subgroup of the absolute Galois group~$\Gal(\bbQ)$ of the rational number field~$\bbQ$, and the reader may wonder if there is not a more abstract description of the splitting of a conjugacy class of involutions in a group~$G$ when intersected with a finite index subgroup~$H$ and its conjugates~$gHg^{-1}$. We have not found such, but offer the following example for contemplation. Take~$G$ to be the alternating group~$\rmA_5$. It contains~15 involutions, all conjugate to~$(12)(34)$. The subgroup~$H$ that is generated by~$(123)$ and~$(12)(45)$ is isomorphic to the symmetric group~$\rmS_3$. The group~$H$ and its conjugate~$gHg^{-1}$ with the element~$g=(14)(25)$ both contain three involutions each, but they only have one of them in common, namely~$(12)(45)$.
\end{example}


\section{Laurent series and other examples}\label{sec:Laurent}

In this section, we will address two questions. First, for any given field~$F$, we will describe the involutory quandle~$\AS(F(\!(t)\!))$ of the field~$F(\!(t)\!)$ of Laurent series with coefficients in~$F$ in terms of the Artin--Schreier quandle~$\AS(F)$. Second, we comment on the class of profinite involutory quandles that are realisable, up to isomorphism, as Artin--Schreier quandles~$\AS(F)$ for some field~$F$. From what we have seen for local and global fields so far, it seems reasonable to conjecture that these are all free, but we will show by example that this is not the case~(see also our Propositions~\ref{prop:anabelian} and~\ref{prop:infinite} for other restrictions).

\begin{theorem}\label{thm:Laurent}
For any field~$F$, there is an isomorphism
\[
\AS(F(\!(t)\!))\cong\AS(F)\times\hatFQ_2\{\sigma,\tau\}\cong\AS(F)\times\hatZ
\]
of profinite involutory quandles, where~$\hatFQ_2\{\sigma,\tau\}$ is the free profinite involutory quandle on a set~$\{\sigma,\tau\}$ with two elements. The quandle structure is given by~$a\rhd b=2a-b$ for~$a,b\in\hatZ$.
\end{theorem}

Before we embark on the proof, let us recall that we can identify~$\hatFQ_2\{\sigma,\tau\}$ with the set of involutions in the free profinite~$2$--group~$\hatF_2\{\sigma,\tau\}$. This group~$\hatF_2\{\sigma,\tau\}$ is the profinite infinite dihedral group~$\widehat{\rmD}_\infty$, and we can identify it with the semi-direct product~$\hatZ\rtimes\bbZ/2$, where the generator of~$\bbZ/2$ acts as~$-1$ on~$\hatZ$. An isomorphism is given by letting a pair~$(a,m)$ in the semi-direct product~$\hatZ\rtimes\bbZ/2$ correspond to~$(\sigma\tau)^a\sigma^m$ in~$\hatF_2\{\sigma,\tau\}$. Once the groups are identified this way, we see that we can thereby also identify the subset~$\hatFQ_2\{\sigma,\tau\}$ of involutions with the set~$\hatZ$, where now~$a\in\hatZ$ corresponds to the pair~$(a,m)$ with~$m$ odd in~$\hatZ\rtimes\bbZ/2$. For~$m$ and~$n$ odd, we get
\begin{align*}
(a,m)(b,n)(a,m)^{-1}
&=(a+(-1)^mb,m+n)((-1)^{-m}(-a),-m)\\
&=(a+(-1)^mb+(-1)^{m+n}(-1)^{-m}(-a),m+n-m)\\
&=(2a-b,n),
\end{align*}
and this calculation shows that the quandle structure we get on the set~$\hatZ$ by transport of structure is given by~\hbox{$a\rhd b=2a-b$}.

\begin{proof}[Proof of Theorem~\ref{thm:Laurent}]
The statement is void if the field~$F$ is not formally real: if the Laurent series field~$F(\!(t)\!)$ can be ordered, so can the subfield~$F$. In particular, we can assume that the field~$F$ has characteristic~$0$. For these fields, it is known that the absolute Galois group~$\Gal(F(\!(t)\!))$ is a semidirect product~(cf.~\cite{vdD+R}); in fact, there is an isomorphism~\hbox{$\Gal(F(\!(t)\!))\cong\hatZ\rtimes\Gal(F)$}~(see~\cite[Prop.~4.1(e)]{Geyer--Jarden}). The multiplication takes the form
\[
(a,g)(b,h)=(a+g(b),gh)
\]
for some~$\Gal(F)$--action on~$\hatZ$. In particular, an element~$(a,g)$ is an involution if and only if~$g=s$ is an involution and~$a+s(a)=0$. 

If an involution~$s$ acted as the identity, then~$a+s(a)=0$ implies~$2a=0$, hence~$a=0$ because~$\hatZ$ is torsionfree. In other words, the element~$s$ in~$\AS(F)=\Inv(\Gal(F))$ would have a unique pre-image in~$\AS(F(\!(t)\!))$, and the same would hold for the conjugates of~$s$. But for every ordering of~$F$, there are two different orderings of~$F(\!(t)\!)$, one where~$t$ is positive and one where~$t$ is negative~(see~\cite[Prop.~VIII.4.11]{Lam}, for instance). Hence, this is not possible, and~$s$ must act non-trivially. In fact, we can be more precise:

\begin{lemma}\label{lem:-1}
Every involution~$s\in\Gal(F)$ acts as~$-1$ on the kernel~$\hatZ$. 
\end{lemma}

If we had an action on the group~$\bbZ$, it would be clear that an involution~$s$ can only act as~$\pm1$, but~$\hatZ$ has more units than these, and we have to argue to exclude any other options. We will do this in detail once we have finished the main argument.

When the involution~$s$ acts as~$-1$, the equation~$a+s(a)=0$ is satisfied by~\textit{all} elements in~$\hatZ$, and we can identify the pre-image~$\{(a,s)\mid a\in\hatZ\}$ of any involution~$s$ canonically with~$\hatZ$. Taken together, we thereby get a bijection between~$\AS(F(\!(t)\!))$ and the product~\hbox{$\AS(F)\times\hatZ$}. The quandle operation is induced from the conjugation in the group, and the computation
\begin{align*}
(a,s)\rhd(b,t)&=(a,s)(b,t)(a,s)\\
&=(a+s(b),st)(a,s)\\
&=(a+s(b)+st(a),sts)\\
&=(a-b+a,sts)\\
&=(2a-b,s\rhd t)
\end{align*}
shows that this proves the statement.
\end{proof}

\begin{proof}[Proof of Lemma~\ref{lem:-1}]
Let~$\overline{F}$ be our chosen algebraic closure of~$F$. Puiseux's theorem states that the algebraic closure of~$\overline{F}(\!(t)\!)$ is the union~$\cup_n\overline{F}(\!(t^{1/n})\!)$, which we will denote by~$\overline{F}(\!(t^{1/\infty})\!)$.  For the field~$F(\!(t)\!)$, we need to be a bit more careful~(see~\cite[Prop.~4.1]{Geyer--Jarden} again). An algebraic closure of~$F(\!(t)\!)$ is the compositum~$\overline{F}\cdot F(\!(t^{1/\infty})\!)$. It contains the subfields~$F(\!(t^{1/\infty})\!)$ and~\hbox{$\overline{F}\cdot F(\!(t)\!)$}, with absolute Galois groups isomorphic to~$\Gal(F)$ and~$\hatZ$, respectively. The latter is normal, and we find that~$\Gal(F(\!(t)\!))\cong\hatZ\rtimes\Gal(F)$ is a semi-direct product. 

To better understand the $\Gal(F)$--action on $\hatZ$, we proceed as follows. A generator~$F_n$ of~$\Gal(\overline{F}\cdot F(\!(t^{1/n})\!)\mid\overline{F}\cdot F(\!(t)\!))\cong\bbZ/n$ acts by~$F_n(t^{1/n})=w_nt^{1/n}$ for a primitive~$n$--th root of unity~$w_n\in\mu_n(\overline{F})$. The elements~$F_n$ and~$w_n$ can be chosen with~\hbox{$(w_{mn})^m=w_n$} for all~$m$ so that the generators~$F_n$ fit together to define a topological generator of~\hbox{$\Gal(\overline{F}\cdot F(\!(t)\!))\cong\hatZ$}. The absolute Galois group~\hbox{$\Gal(F)$} acts on the coefficients of the Laurent series. On the normal subgroup~\hbox{$\Gal(\overline{F}\cdot F(\!(t)\!))$} of~$\Gal(F(\!(t)\!))$, it acts by conjugation, but the computation
\[
sF_ns^{-1}(t^{1/n})=sF_n(t^{1/n})=s(w_nt^{1/n})=s(w_n)t^{1/n}
\]
shows that this conjugation action is also just an action on the coefficients. In summary, we see that the~$\Gal(F)$--action on~$\hatZ$ is determined by the~$\Gal(F)$--action on the group~$\mu_\infty(\overline{F})$ of roots of unity in the algebraic closure~$\overline{F}$.

It remains for us to check that all involutions in the absolute Galois group~$\Gal(F)$ act on the group~$\mu_\infty(\overline{F})$ of roots of unity in the algebraic closure~$\overline{F}$ via the inversion~\hbox{$w\mapsto w^{-1}$}. We argue one prime at a time. If the prime~$p$ is odd, then all primitive~$p$--th power roots of unity are totally imaginary, and they cannot be fixed by an involution~$s$. As there is only one involution in the cyclic group~$(\bbZ/p^e)^\times$, this forces~\hbox{$s(w)=w^{-1}$} for all~\hbox{$w\in\mu_{p^e}(\overline{F})$}. At the even prime~$2$, when~$w\in\mu_{2^e}(\overline{F})$, we must have~$s(w)=w^{-1}$ for~$e=1$ trivially and for~$e=2$ by a similar reasoning as above: the group~$(\bbZ/4)^\times$ is cyclic. However, for~$e\geqslant3$, there are two other options for us to consider, namely~\hbox{$s(w)=-w$} and~$s(w)=-w^{-1}$. The first would imply that we have~\hbox{$s(w^2)=w^2$}, which we can rule out by induction on~$e$, as~\hbox{$w^2\in\mu_{2^{e-1}}(\overline{F})$}. The second would imply that~$w-w^{-1}$ were fixed by~$s$, but this element is totally imaginary. That leaves us with~\hbox{$s(w)=w^{-1}$} for all~$w\in\mu_{2^e}(\overline{F})$ as well.
\end{proof}

\begin{example}
As the Artin--Schreier quandle~$\AS(\bbR)$ is a singleton, Theorem~\ref{thm:Laurent} immediately gives that the Artin--Schreier quandle~$\AS(\bbR(\!(t)\!))$ is free, isomorphic to~$\hatFQ_2\{\sigma,\tau\}$. By intersecting two real closed fields, Geyer~\cite{Geyer} has constructed other formally real fields whose absolute Galois group is infinite dihedral,~i.e., isomorphic to the free~\hbox{pro-2-group}~\hbox{$\hatF_2\{\sigma,\tau\}$} on two generators. Of course, the correspoding Artin--Schreier quandle~\hbox{$\hatFQ_2\{\sigma,\tau\}$} is free on two generators as well.
\end{example}

We conclude with an example showing that Artin--Schreier quandles need not be free.

\begin{example}
Applying Theorem~\ref{thm:Laurent} twice, we see that
\[
\AS(\bbR(\!(t_1)\!)(\!(t_2)\!))\cong\hatFQ_2\{\sigma_1,\tau_1\}\times\hatFQ_2\{\sigma_2,\tau_2\},
\]
and this profinite involutory quandle is not free. In fact, it is not even isomorphic to a quandle of the form~$\hatFQ_2(S)$ after forgetting the topologies, i.e., algebraically. To see this, we can consider the group of transvections~(or displacements) of these quandles, introduced by Joyce~\cite[Sec.~5]{Joyce}: for a given quandle~$Q$, the subgroup~$\Tr(Q)\leqslant\Aut(Q)$ is the subgroup generated by the differences ${\lambda_a}^{\!\!\!-1}\lambda_b$ for all~\hbox{$a,b\in Q$}. If $S$ has at most one element, this group is trivial. As $\hatFQ_2\{\sigma,\tau\}\cong\hatZ$ with $a\rhd x=2a-x$, we can identify the group~$\Tr(\hatFQ_2\{\sigma,\tau\})$ with the group of maps~$\hatZ\to\hatZ$ of the form~\hbox{$x\mapsto x+2c$}, and this group is isomorphic to~$\hatZ$ again. As~$\Tr(Q\times R)\cong\Tr(Q)\times\Tr(R)$ by~\cite[Lem.~5.1]{Kai--Tamaru}, we have~\hbox{$\Tr(\AS(\bbR(\!(t_1)\!)(\!(t_2)\!)))\cong\hatZ^2$} for the quandle under consideration. It remains for us to argue that the groups~$\Tr(\hatFQ_2(S))$ are never of this form. From what we have computed above, this is already clear for~$|S|<3$. But if~$S$ contains three different elements~$s$,~$t$,~and~$u$, the group~$\Tr(\hatFQ_2(S))$ is not even abelian, as the elements~${\lambda_s}^{\!\!\!-1}\lambda_u$ and ${\lambda_t}^{\!\!\!-1}\lambda_u$ do not commute~(check on~$u$).
\end{example}


\section*{Acknowledgements} 

I thank Jayanta Manoharmayum for helpful discussions on absolute Galois groups, and the anonymous referee for exceptionally detailed and constructive comments that led to improvements in several aspects of this paper.



\vfill

School of Mathematical and Physical Sciences, University of Sheffield, Hicks Building, Hounsfield Road, Sheffield S3 7RH, UK\\
\href{mailto:m.szymik@sheffield.ac.uk}{m.szymik@sheffield.ac.uk}

Fakultät für Mathematik, Ruhr-Universit\"at Bochum, Universit\"atsstraße 150, 44801 Bochum, Germany\\
\href{mailto:markus.szymik@rub.de}{markus.szymik@rub.de}

\end{document}